\documentclass[11pt]{article}
\usepackage[T1]{fontenc}
\usepackage{mathptmx} 
\usepackage{amsmath,amssymb,amsthm,subfigure}
\usepackage[dvips]{epsfig}
\usepackage{titlesec}
\usepackage{setspace}
\usepackage{amsmath,amssymb,subfigure}
\usepackage[dvips]{epsfig}
\usepackage{fancyhdr}

\usepackage[margin=1.0in]{geometry}

\usepackage{indentfirst}

\makeatletter \@addtoreset{equation}{section} \makeatother

\AtBeginDocument{%
\setlength\abovedisplayskip{8pt}
\setlength\belowdisplayskip{8pt}}

\begin{document}

\fancyhf{}
\renewcommand{\headrulewidth}{0pt}

\pagestyle{empty}

\newtheorem{theo}{Theorem}[section]
\newtheorem{lem}[theo]{Lemma}
\newtheorem{cor}[theo]{Corollary}
\newtheorem{pro}[theo]{Proposition}
\newtheorem{prop}[theo]{Property}
\newtheorem{defi}[theo]{Definition}

\newcommand{\no}{\nonumber}
\newcommand{\p}{\partial}
\newcommand{\lb}{\lambda}
\newcommand{\vu}{\vec{u}}
\newcommand{\tvu}{\vec{\tilde{u}}}
\newcommand{\e}{\epsilon}
\newcommand{\ra}{\rightarrow}
%\renewenvironment{proof}{\noindent {\bf Proof: }}{\hfill $\Box$}

%\makeatletter
%\renewenvironment{proof}[1][\proofname]{%
%  \par\pushQED{\qed}\normalfont%
%  \topsep6\p@\@plus6\p@\relax
%  \trivlist\item[\hskip\labelsep\bfseries#1\@addpunct{:}]%
%  \ignorespaces
%}{%
%  \popQED\endtrivlist\@endpefalse
%}
%\makeatother

\begin{center}
{\bf \Large  Extremal Functions of Forbidden Multidimensional Matrices} 

\vspace*{.3in}

{Jesse T. Geneson \\
Department of Mathematics, MIT, Cambridge, MA 02139, USA \\ geneson@math.mit.edu

\bigskip

and

\bigskip

Peter M. Tian \\ 
Department of Mathematics, Harvard University, Cambridge, MA 02138, USA \\ ptian@college.harvard.edu }
\end{center}

\vspace*{.5in}

\begin{abstract}

Pattern avoidance is a central topic in graph theory and combinatorics. Pattern avoidance in matrices has applications in computer science and engineering, such as robot motion planning and VLSI circuit design. A $d$-dimensional zero-one matrix $A$ avoids another $d$-dimensional zero-one matrix $P$ if no submatrix of $A$ 
can be transformed to $P$ by changing some ones to zeros. A fundamental problem is to study the maximum number of nonzero entries in 
a $d$-dimensional $n \times  \cdots \times n$ matrix that avoids $P$. This maximum number, denoted by $f(n,P,d)$,  is called the extremal function.
    
We advance the extremal theory of matrices in two directions. The methods that we use come from combinatorics, probability, and analysis.  
Firstly, we obtain non-trivial lower and upper bounds on $f(n,P,d)$ when $n$ is large for every $d$-dimensional block permutation matrix $P$. We establish the tight bound $\Theta(n^{d-1})$ on $f(n,P,d)$ for every $d$-dimensional tuple permutation matrix $P$. This tight bound has the lowest possible order that an extremal function of a nontrivial matrix can ever achieve. 
Secondly, we show that $f(n,P,d)$ is super-homogeneous for a class of matrices $P$. We use this super-homogeneity to show that 
the limit inferior of the sequence $\{ {f(n,P,d) \over n^{d-1}}\}$ has a lower bound $2^{\Omega(k^{1/ d})}$ for a family of $k \times \cdots \times k$ permutation matrices $P$. We also improve the upper bound on the limit superior from $2^{O(k \log k)}$ to $2^{O(k)}$ for all $k \times \cdots \times k$ 
permutation matrices and show that the new upper bound also holds for tuple permutation matrices. 

\end{abstract}

\newpage

\setcounter{page}{1}
\pagestyle{plain}
\setcounter{section}{0}

%\fancyfoot[C]{\thepage}

\section{Introduction}

Pattern avoidance is a central topic in graph theory and combinatorics \cite{B, Fox, HM, Ki, KM, MT, SA, S, St}. In this paper, we study $d$-dimensional matrices (or arrays) with only two distinct entries, $0$ and $1$,  that avoid certain patterns. We are interested in the extremal problem of finding the largest number of ones in these matrices. 

The extremal theory of matrices was initiated in papers \cite{BG, F, M} around 1990 to study specific problems arising in computational and discrete geometry. Mitchell produced an algorithm for computing a shortest rectilinear path avoiding rectilinear obstacles in the plane \cite{M}. He showed that the complexity of this algorithm is bounded above by the solution of the extremal problem of certain matrices. Bienstock and Gy\"ori \cite{BG} found an upper bound on the solution of the extremal problem,  thus bounding the complexity of the algorithm. Mitchell's algorithm can be directly applied to motion planning in robotics and wire routing in VLSI circuit design \cite{LYW}. F\"{u}redi \cite{F} used the extremal problem to derive an upper bound on Erd\H{o}s-Moser \cite{EM} problem of determining the maximum number of unit distances in a convex polygon.

Matrix extremal problems also find applications to graph theory and other areas of combinatorics. If we associate $2$-dimensional $0$-$1$ matrices with ordered bipartite graphs by relating rows and columns to the two ordered partite sets of vertices and  interpreting ones as edges, then this extremal problem can be viewed as the Tur\'an extremal problem for ordered bipartite graphs \cite{PT}.  The most recent surge in interest in the extremal theory of matrices is due to the resolution of the Stanley-Wilf conjecture in enumerative combinatorics using the extremal problem of matrices \cite{Kl, MT}.

We denote a $d$-dimensional $n_1 \times \cdots \times n_d$ matrix by $A = \left(a_{i_1, \ldots , i_d}\right)$, where $1 \le i_l \le n_{\ell}$ for $\ell = 1, 2, \ldots , d$. Matrix $A$ is called a $ 0$-$1$  matrix if all its entries are either $0$ or $1$. A multidimensional matrix is also called a multidimensional array in computer science. We may view a $d$-dimensional $0$-$1$ matrix $A = \left( a_{i_1, \ldots , i_d} \right)$ geometrically as a $d$-dimensional rectangular box of lattice points with coordinates $(i_1, \ldots, i_d)$.  An $\ell$-cross section of matrix $A$ is the set of all the entries $a_{i_1,\ldots , i_d}$ whose $\ell^{\text{th}}$ coordinates 
have the same value. An $\ell$-row of matrix $A$ is the collection of all the entries $a_{i_1,\ldots , i_d}$ whose coordinates other than the 
$\ell^{\text{th}}$ coordinate have fixed values.

A $d$-dimensional $k  \times \cdots \times k$ zero-one matrix is a permutation matrix if each of its $\ell$-cross sections contains exactly one nonzero entry for every $\ell=1, \ldots , d$.
The Kronecker product of two $d$-dimensional $0$-$1$ matrices $M$ and $N$, denoted by $M \otimes N$, is a 
$d$-dimensional matrix obtained by replacing each 1-entry of $M$ with a copy of $N$ and each 0-entry of $M$ with a zero matrix the same size as $N$. 

We say that a $d$-dimensional $0$-$1$ matrix $A$ contains another $0$-$1$ matrix $P$ if $A$ has a submatrix that can be transformed into $P$ by changing any number of ones to zeros. Otherwise, $A$ is said to avoid $P$.  Denote by $f(n, P, d)$ the maximum number of ones in a $d$-dimensional $n \times \cdots \times n$ zero-one matrix that avoids a given $d$-dimensional $0$-$1$ matrix $P$. We are interested in the asymptotic behavior of the extremal function $f(n,P,d)$ for large $n$.
  
A motivation to study the extremal problem of multidimensional matrices comes from graph theory. A $2$-dimensional $n \times n$ zero-one matrix 
can be represented as an ordered bipartite graph with partite sets of size $n$. Similarly, a $d$-dimensional $0$-$1$ matrix corresponds  
to  an ordered $d$-partite, $d$-uniform hypergraph. The extremal function $f(n, P, d)$ is the maximum 
number of hyperedges in an ordered $d$-partite, $d$-uniform hypergraph where each partite set has $n$ vertices. Avoiding a given pattern $P$ amounts to avoiding the
corresponding ordered hypergraph.

It is easy to obtain trivial lower and upper bounds on $f(n,P,d)$.
\begin{pro}
\label{Easy}
If $P$ is a $0$-$1$ matrix that contains at least two ones, then
$n^{d-1} \leq f(n,P,d) \leq n^{d}$.
\end{pro}
\begin{proof}
We can always choose a $d$-dimensional  $n \times \cdots \times n$ zero-one matrix $A$, with $1$-entries on a single $\ell$-cross section for some $\ell$ 
and $0$-entries elsewhere, such that $A$ avoids $P$. Matrix $A$ has exactly $n^{d -1}$ ones, so the left inequality follows.

The right inequality follows from the fact that every $d$-dimensional  $n \times  \cdots \times n$ zero-one matrix has a total of $n^d$ entries and that the matrix hence has at most $n^d$ ones.
\end{proof}

The upper bound in Proposition \ref{Easy} is one order higher than the lower bound. The main problem is to improve the lower and upper bounds on $f(n,P,d)$ so that their orders are as close as possible.

The two-dimensional case of $d=2$ has been well studied.  F\"{u}redi and Hajnal conjectured that $f(n,P,2)=O(n)$ for all permutation matrices $P$ \cite{FH}.  Klazar showed that this conjecture implies the Stanley-Wilf conjecture \cite{Kl}. Marcus and Tardos proved the  F\"{u}redi and Hajnal conjecture \cite{MT} and hence settled the Stanley-Wilf conjecture.  Keszegh conjectured that $f(n, P, 2)=O(n)$ for all tuple permutation matrices $P$ \cite{K}. Geneson proved that the conjecture is true \cite{G}.

When $R$ is a $k_1 \times k_2$ matrix of all ones, the extremal problem for $f(n, R, 2)$ is the matrix version of the classical 
Zarankiewicz problem.  K\H{o}v\'{a}ri, S\'{o}s, and Tur\'{a}n found an upper bound $O(n^{2 - {\max({k_1, k_2}) \over k_1 k_2}})$ on $f(n,R,2)$ \cite{KST}.
A lower bound $\Omega(n^{2 - {k_1 + k_2 - 2 \over k_1 k_2 -1}})$ was also known \cite{ES}.  Hesterberg extended these bounds to $f(n, P \otimes R, 2)$ where $P$ is a permutation matrix \cite{HP, H}.

Pach and Tardos showed that $f(n, P, 2)$ is super-additive in $n$ \cite{PT}.
By Fekete's Lemma on super-additive sequences \cite{Fekete}, the sequence $\{ {f(n, P, 2) \over n} \}$ is convergent. The limit is known as the F\"uredi-Hajnal limit and has recently attracted great attention \cite{C, Fox, Fox2}. Cibulka \cite{C} showed that this limit is always at least $2(k-1)$ when $P$ is a $k \times k$ permutation matrix and that the limit is exactly $2(k-1)$ when $P$ is the identity matrix. 
Fox showed that the F\"uredi-Hajnal limit has a lower bound $2^{\Omega((k^{1 \over 2})}$ for a family of $k \times k$ permutation matrices \cite{Fox2}.
Marcus and Tardos \cite{MT} showed that this limit has an upper bound  $2^{O(k \log k)}$ for every $k \times k$ permutation matrix $P$, and Fox \cite{Fox} improved this upper bound to $2^{O(k)}$. 

Little has been done on the multidimensional case. Klazar and Marcus \cite{KM} studied the extremal function when the $d$-dimensional matrix $P$ is a permutation matrix of size 
$k \times \cdots \times k$ and found $f(n,P,d)=O(n^{d-1})$, generalizing the $d=2$ results \cite{ MT}. In particular, they showed that ${f(n, P, d) \over n^{d-1}} = 2^{O(k \log k)}$, which is the  
multidimensional generalization of the Marcus and Tardos upper bound on the F\"uredi-Hajnal limit \cite{ MT}. 

In this paper, we advance the extremal theory of matrices in two directions. In the first direction,
we study the extremal functions $f(n, R, d)$ and $f(n, P \otimes R, d)$, where $P$ is a permutation matrix and $R$ is a
$k_1 \times  \cdots \times k_d$ matrix of ones only; matrix $P \otimes R$ is called a block permutation matrix.
We show that both $f(n, R, d)$ and $f(n, P \otimes R, d)$ have a lower bound
$\Omega(n^{d - \beta})$ and 
an upper bound $O(n^{d - \alpha})$, where  $\alpha = {\max({k_1, \ldots , k_d}) \over  k_1 \cdot k_2 \cdots k_d }$ and $\beta = {k_1+ \cdots +k_d - d \over k_1 \cdot k_2  \cdots k_d-1}$. These 
bounds significantly improve the trivial ones given in Proposition \ref{Easy}. 

Both the lower bound $\Omega(n^{d - \beta})$ and upper bound $O(n^{d - \alpha})$  are in the same order as 
$\Theta(n^{d - 1})$ whenever $\alpha = \beta$. This is exactly when only one of $k_1, \ldots , k_d$ of 
$R$ is larger than $1$; the corresponding $P \otimes R$ is a tuple permutation matrix.

When $\alpha \neq \beta$, our results on $f(n,R,d)$ generalize the K\H{o}v\'{a}ri-S\'{o}s-Tur\'{a}n upper bound \cite{KST} from two dimensions to higher dimensions. Our results on $f(n, P \otimes R, d)$ also extend Hesterberg's results \cite{HP} from $d=2$ to $d > 2$. 

When $\alpha = \beta$, our result that $f(n, P, d) = \Theta(n^{d - 1})$ for every $d$-dimensional tuple permutation matrix $P$, on one hand, generalizes Geneson's result \cite{G}
from $d=2$ to $d \geq 2$. On the other hand, even when $d=2$ our ideas improve some key calculations in Geneson's paper \cite{G}.
These improvements are vital in our derivation of a new upper bound on the limit superior of the sequence $\{ {f(n,P,d) \over n^{d-1}}\}$ that we discuss below. 

The importance of our result $f(n, P, d) = \Theta(n^{d - 1})$ for every $d$-dimensional tuple permutation matrix $P$ lies in the fact that, in view of Proposition \ref{Easy}, $\Theta(n^{d - 1})$ is the lowest possible order for the extremal function of any nontrivial $d$-dimensional matrix. 

In the second direction, we study the limit inferior and limit superior of the sequence $\{ {f(n,P,d) \over n^{d-1}} \}$ where $P$ satisfies $f(n, P, d) = \Theta(n^{d - 1})$. These are the multidimensional analogues of the F\"uredi-Hajnal limit. We show that the limit inferior is at least $d(k-1)$ for $k \times \cdots \times k$ permutation matrices, 
generalizing Cibulka's result \cite{C} from $d=2$ to all $d \ge 2$. 

We observe that $f(n, P, d)$ is super-homogeneous in higher dimensions, i.e., $f(sn, P, d) \geq K s^{d-1} f(n, P, d)$ for some positive constant $K$.
This super-homogeneity is key to our proof that 
the limit inferior of $\{ {f(n,P,d) \over n^{d-1}} \}$ has a lower bound $2^{\Omega(k^{1 / d})}$ for a family of $k \times \cdots \times k$ permutation matrices, generalizing Fox's result \cite{Fox2} from $d=2$ to $d \geq 2$.

Finally, we show that the limit superior of  the sequence $\{ {f(n,P,d) \over n^{d-1}} \}$ is bounded above by $ 2^{O(k)}$ for all $k \times \cdots \times k$ permutation matrices $P$.
This is a substantial improvement of Klazar and Marcus upper bound $2^{O(k \log k)}$ for $d > 2$ in 
paper \cite{KM} and it also generalizes Fox's bound $2^{O(k)}$ on the F\"uredi-Hajnal limit in two dimensions \cite{Fox}. We further show that this upper bound $2^{O(k)}$ is also true for every tuple permutation matrix $P$, which is a new result even for $d=2$. We are able to extend the new upper bound from permutation matrices to tuple permutation matrices mainly because of our improvement of Geneson's approach as mentioned above.

The rest of the paper is organized as follows. In Section 2, we study $f(n,P,d)$ when $P$ is a block permutation matrix but not a tuple permutation matrix. The more difficult case when $P$ is a tuple permutation matrix is analyzed in Section \ref{tuple}. 
In Section 4, we study the limit inferior and limit superior of the sequence $\{ {f(n,P,d) \over n^{d-1}}\}$ for permutation and tuple permutation matrices $P$. We conclude the paper and discuss our future directions in Section \ref{conclusion}.

\section{Block permutation matrices}
\label{Block}
\bigskip

In this section, we study the extremal function of a variant of $d$-dimensional permutation matrices. 
We are interested in the forbidden matrices which can be written as the Kronecker product of a $d$-dimensional permutation 
matrix and a $d$-dimensional matrix of $1$-entries only.

Let $R^{k_1,...,k_d}$ be the $d$-dimensional $k_1 \times \cdots \times k_d$ matrix of all ones. We 
study lower and upper bounds on the extremal function of block permutation matrix $P \otimes R^{k_1,...,k_d}$, 
where $P$ is a $d$-dimensional permutation matrix.

We first study the extremal function of  $R^{k_1, \ldots  , k_d}$. We use the probabilistic method to  obtain a lower 
bound on $f(n,R^{k_1,...,k_d},d)$. When $d=2$, this lower bound is classical \cite{ES}.
\begin{theo}
\label{lowerbound}
If $k_1 \cdot k_2 \cdots k_d > 1$, then
$f(n,R^{k_1, \ldots , k_d},d) = \Omega \left( n^{d- \beta(k_1, k_2, \ldots , k_d)} \right)$, where $\beta = {k_1+ \cdots +k_d - d \over k_1 \cdot k_2 \cdots k_d-1}$.
\end{theo}
\begin{proof} Let each entry of a $d$-dimensional $n \times \cdots \times n$ zero-one matrix 
$A$ be chosen to be $1$ with probability $p=n^{- \beta(k_1, \ldots , k_d)}$ and $0$ with probability $1 - p$. 
The expected number of $1$-entries in $A$ is $pn^d$. There are ${n \choose k_1} \cdot {n \choose k_2} \cdots {n \choose k_d}$ 
possible copies of $R^{k_1, \ldots , k_d}$ in matrix $A$ and each has a probability of $p^{k_1 \cdot k_2 \cdots k_d}$ of occurring.
The expected number of copies of $R^{k_1, \ldots , k_d}$ in $A$ is
$$ {n \choose k_1} \cdot {n \choose k_2} \cdots {n \choose k_d}p^{k_1 \cdot k_2 \cdots k_d} \le Cn^{k_1+ \cdots + k_d}p^{k_1 \cdot k_2 \cdots k_d} \ ,$$
where, since at least one of $k_1$, $\ldots$ , $k_d$ is greater than one,  $C$ is a positive constant less than 1. 

Let $A'$ be the matrix formed by changing a single $1$-entry in each copy of $R^{k_1, \ldots , k_d}$ in 
$A$ to a 0-entry. Then $A'$ avoids $R^{k_1,...,k_d}$ and the expected number of $1$-entries in $A'$ is at least
$pn^d- Cn^{k_1+k_2+ \cdots + k_d}p^{k_1 \cdot k_2 \cdots k_d}=(1 - C) \ n^{d-\beta(k_1, k_2, \ldots, k_d) }$.
 As a consequence, there exists some matrix $A'$ that avoids $R^{k_1, \ldots , k_d}$ 
and has at least so many $1$-entries.
\end{proof}

We now obtain an upper bound on the extremal function of $R^{k_1, \ldots , k_d}$. When $d=2$,  this upper bound is due to K\H{o}v\'{a}ri, S\'{o}s, and Tur\'{a}n \cite{KST}.
\begin{theo}
\label{upperbound}
$f(n,R^{k_1, \ldots , k_d},d)=O(n^{d- \alpha(k_1, \ldots , k_d)})$, where $\alpha = {\max({k_1, \ldots , k_d}) \over  k_1 \cdot k_2 \cdots k_d }$.
\end{theo}
\begin{proof}
We prove the theorem by induction on $d$. The base case of $d=1$ is trivial. Assuming that 
$f(n, R^{k_1, \ldots, k_{d-1}},d-1)=O(n^{d-1- \alpha(k_1, \ldots, k_{d-1})})$ 
for some $d \geq 2$, we show that $f(n,R^{k_1, \ldots , k_d},d)=O(n^{d- \alpha(k_1, \ldots , k_d)})$.

Throughout the proof, we let  $A = (a_{i_1, \ldots , i_{d}})$ be a $d$-dimensional $n \times \cdots \times n$ matrix  
that avoids $R^{k_1, \ldots , k_d}$ and has the maximum number, $f(n,R^{k_1, \ldots , k_d},d)$, of  ones. We need the following  lemma on the number 
of $d$-rows that have $1$-entries in each of predetermined $k_d$ $d$-cross sections.

\begin{lem}
\label{d-row}
For any set of $k_d$ $d$-cross sections of $A$, there are $O\left(n^{d-1-\alpha(k_1, \ldots , k_{d-1})} \right)$ $d$-rows in $A$ which contain a 1-entry in each of these $d$-cross sections.

\end{lem}
\begin{proof}
Let the $d^{\text{th}}$ coordinates of these $d$-cross sections be $\ell_1, \ldots , \ell_{k_d}$. Define a ($d-1$)-dimensional
$n\times \cdots \times n$ matrix $B=(b_{i_1, \ldots , i_{d-1}})$ such that
$b_{i_1, \ldots , i_{d-1}}=1$ if $a_{ i_1, \ldots , i_{d-1},\ell_1}= \cdots =a_{i_1, \ldots , i_{d-1}, \ell_{k_d}}=1$ 
and $b_{i_1, \ldots , i_{d-1}}=0$ otherwise.

We claim that matrix $B$ must avoid $R^{k_1, \ldots , k_{d-1}}$. Suppose to the contrary that 
$B$ contains $R^{k_1, \ldots , k_{d-1}}$. Let $e_1, \ldots , e_{k_1 \cdot k_2 \cdots k_{d-1}}$ be all the $1$-entries 
in $B$ that represent $R^{k_1,...,k_{d-1}}$. By the construction of $B$,
there are $k_d$ nonzero entries  with coordinates $(x_1, \ldots , x_{d-1},\ell_1), \ldots , (x_1, \ldots , x_{d-1}, \ell_{k_d})$ in $A$ corresponding to
each $e_i$ with coordinates $(x_1, \ldots , x_{d-1})$ in $B$. All these $k_1 \cdot k_2 \cdots k_d$ nonzero entries 
form a copy of $ R^{k_1, \ldots , k_{d}}$ in $A$, a contradiction. Thus $B$ must avoid 
$R^{k_1, \ldots , k_{d-1}}$ and  by our inductive assumption, $B$ must have 
$O(n^{d-1-\alpha(k_1, \ldots , k_{d-1})})$ ones. The result follows.
\end{proof}

Suppose all the $d$-rows of $A$ have $r_1, \ldots , r_{n^{d-1}}$  non-zero entries, respectively. 
Counting the total number of sets of $k_d$ nonzero entries in the same $d$-row in two different ways yields
\begin{equation}
\label{two}
\sum_{i=1}^{n^{d-1}}{r_i \choose k_d} = {n \choose k_d}O\left( n^{d-1-\alpha(k_1, \ldots , k_{d-1})} \right) \ ,
\end{equation}
where we use Lemma \ref{d-row} to obtain the right hand side.

Matrix $A$ avoids $R^{k_1, \ldots, k_d}$ and has the largest possible number of $1$-entries, so $r_i \geq k_d - 1$ for $1 \leq i \leq  n^{d-1}$. 
Since ${r \choose k}$ is a convex function of $r$ for $r \geq k - 1$,  we apply Jensen's inequality to obtain
\begin{eqnarray*}
\sum_{i=1}^{n^{d-1}}{r_i \choose k_d} \ge n^{d-1} {{1\over n^{d-1}}\sum_{i=1}^{n^{d-1}} r_i \choose k_d} 
= n^{d-1}{{1\over n^{d-1}}f(n,R^{k_1, \ldots , k_d},d) \choose k_d} \ ,
\end{eqnarray*}
where, in the equality, we use the assumption that $A$ has $f(n,R^{k_1, \ldots , k_d}, d)$ total $1$-entries.
Substituting this into equation (\ref{two}) yields
$$n^{d-1}{{1\over n^{d-1}}f(n,R^{k_1, \ldots , k_d},d) \choose k_d} = {n \choose k_d}O\left(n^{d-1-\alpha(k_1,\ldots , k_{d-1})} \right) \ ,$$
\noindent
which together with ${n \choose k}=\Theta(n^k)$ gives
$$n^{d-1} \left({1\over n^{d-1}}f(n,R^{k_1, \ldots , k_d},d)\right)^{k_d} = O\left( n^{ k_d}\cdot n^{d-1-\alpha(k_1,\ldots , k_{d-1})} \right) \ . $$ 
This implies
$$f\left( n,R^{k_1, \ldots , k_d},d \right) = O \left( n^{d-{\alpha(k_1, \ldots , k_{d-1}) \over k_d }} \right) \ . $$
Similarly, we have
$$f(n,R^{k_1, \ldots , k_d},d) = O\left(n^{d-{\alpha(k_2, \ldots , k_d) \over k_1}}\right) \ .$$ 
Note that $\max\left({\alpha(k_2, \ldots , k_d) \over k_1}, {\alpha(k_1, \ldots , k_{d-1}) \over k_d}\right)=\alpha(k_1, \ldots , k_d)$. 
Thus taking the smaller of the two upper bounds gives
$$f(n,R^{k_1, \ldots , k_d},d) = O\left(n^{d-\alpha(k_1, \ldots , k_d)}\right) $$ 
which completes the inductive step, and thus Theorem \ref{upperbound} is proved.
\end{proof}

We make the following observation on $\alpha(k_1, \ldots , k_d)$ and $\beta(k_1, \ldots , k_d)$.
\begin{pro}
\label{alpha}
 Suppose $d > 1$ and $k_1, \ldots , k_d$ be positive integers such that $k_1 \cdot k_2 \cdots k_d > 1$. If only one  of $k_1, \ldots , k_d$ is greater than 1, then
$\alpha(k_1, \ldots , k_d) = \beta(k_1, \ldots , k_d) = 1 $.
Otherwise,
$0 < \alpha(k_1, \ldots , k_d) <  \beta(k_1, \ldots , k_d) < 1 $.
\end{pro}

We omit the proof since it is straightforward. Proposition \ref{alpha} implies that the lower bound of Theorem \ref{lowerbound} and the upper bound of Theorem \ref{upperbound} are significant improvements of the bounds in Proposition \ref{Easy}.

We now study the extremal function of the Kronecker product $P \otimes R^{k_1, \ldots , k_d}$, where $P$ is a $d$-dimensional permutation matrix. We show that the extremal functions of $P \otimes R^{k_1, \ldots , k_d}$ and $R^{k_1, \ldots , k_d}$ share the same lower and upper bounds.

\noindent
\begin{theo}
\label{super}
If $P$ is a $d$-dimensional permutation matrix and at least two of $k_1, \ldots , k_d$ are greater than 1, then there exist constants $C_1$ and $C_2$ such that for all $n$,
\begin{equation}
\label{block}
C_1 n^{d - \beta(k_1, \ldots , k_d)} \leq f(n,P\otimes R^{k_1, \ldots , k_d}, d) \leq C_2 n^{d - \alpha(k_1, \ldots , k_d)}
\end{equation}
\end{theo}
\noindent
\begin{proof} We first have 
\begin{equation}
\label{R}
f(n, R^{k_1, \ldots , k_d}, d) \leq  f(n,P\otimes R^{k_1, \ldots , k_d}, d).
\end{equation}
 This follows from the fact that any matrix that avoids $R^{k_1, \ldots , k_d}$ must also avoid
$P\otimes R^{k_1, \ldots , k_d}$. The left inequality of (\ref{block}) is then the result of (\ref{R}) and Theorem \ref{lowerbound}.

To prove the right inequality of (\ref{block}), we follow Hesterberg's idea for the $2$-dimensional case \cite{H} to 
estimate $f(n,P\otimes R^{k_1, \ldots , k_d}, d)$ first for $n = c^m$, where $m$ is an arbitrary positive integer and $c$ 
is a positive integer to be determined, and then for all other positive integers $n$.  

We make use  of the upper bound in Theorem \ref{upperbound}
\begin{equation}
\label{Kconstant}
f(n, R^{k_1, \ldots , k_d}, d) \leq g(n) \ , 
\end{equation}
where $g(n) = K n^{d - \alpha(k_1, \ldots , k_d)}$ for some positive constant $K$, and claim that
\begin{equation}
\label{cm}
 f(c^m,P\otimes R^{k_1, \ldots , k_d}, d) \leq  2 c^d g(c^m).
\end{equation}

We justify the claim by induction. The base case of $m=0$ is trivially true. Suppose that 
\begin{equation}
\label{indu}
f(n,P\otimes R^{k_1, \ldots , k_d},d) \le 2c^d g(n) 
\end{equation}
for $n = c^m$. We show that 
$f(cn,P\otimes R^{k_1, \ldots , k_d},d)\le 2c^d g(cn)$.

Let $A$ be a $d$-dimensional $cn\times \cdots \times cn$ matrix avoiding $P\otimes R^{k_1, \ldots , k_d}$ with $f(cn, P \otimes R^{k_1, \ldots , k_d}, d)$ total $1$-entries. We divide $A = (a_{i_1, \ldots , i_d})$ into $c^d$ 
disjoint submatrices of size $n \times \cdots \times n$. We label these submatrices by $S(i_1, \ldots , i_d) = (s_{j_1, \ldots , j_d})$, where 
$$s_{j_1, \ldots , j_d} =a_{j_1+n(i_1-1), \ldots , j_d+n(i_d-1)} \ .$$
These are called $S$ submatrices throughout the paper. 

Let $C$ be the $d$-dimensional $c \times \cdots \times c$  matrix such that $c_{i_1, \ldots , i_d}=1$ 
if submatrix $S(i_1, \ldots , i_d)$ of $A$ contains $R^{k_1, \ldots , k_d}$ and that $c_{i_1, \ldots , i_d}=0$ otherwise.  
Since any two $1$-entries of the  permutation matrix $P$ differ in all coordinates,  $C$ must avoid $P$ or else $A$ contains $P\otimes R^{k_1, \ldots , k_d}$.

We can classify all the $S$ submatrices of $A$ into two classes.

\medskip

\noindent
{\bf Case 1: $S$ contains $R^{k_1, \ldots , k_d}$} 

Since $C$ avoids $P$, there are at most $f(c,P,d)$ such $S$ submatrices. Clearly each $S$ 
submatrix must avoid $P \otimes R^{k_1, \ldots , k_d}$, so it has at most $f(n,P\otimes R^{k_1, \ldots , k_d},d)$ 
$1$-entries. There are at most $f(c,P,d)f(n,P\otimes R^{k_1, \ldots , k_d},d)$ $1$-entries from this type of $S$ submatrices.

\medskip

\noindent
{\bf Case 2: $S$ avoids $R^{k_1, \ldots , k_d}$} 

There are at most $c^d$  such submatrices in total. Each has at most $f(n,R^{k_1, \ldots , k_d},d)$ $1$-entries. 
There are at most $c^df(n,R^{k_1, \ldots , k_d},d)$ $1$-entries from the second type of $S$ submatrices.

\medskip

\noindent
Summing the numbers of $1$-entries in both cases gives 
\begin{equation}
f(cn,P\otimes R^{k_1,...,k_d},d)\le f(c,P,d)f(n,P\otimes R^{k_1, \ldots , k_d},d)+c^df(n,R^{k_1, \ldots , k_d},d)  . \no
\end{equation}
On the right hand side of the inequality, $f(n,P\otimes R^{k_1, \ldots , k_d},d)$ has an upper bound $2c^d g(n)$ 
because of the inductive assumption (\ref{indu}) and $f(n,R^{k_1, \ldots , k_d},d)$ has an upper bound $g(n)$ by (\ref{Kconstant}). Since $f(c,P,d) = O(c^{d-1})$ 
for any permutation matrix $P$ \cite{KM},  
there exists a constant $L$ such that $f(c,P,d) \le Lc^{d-1}$. Because at least two of $k_1, k_2, \ldots , k_d$ are greater than $1$, it follows from Proposition \ref{alpha}
that $\alpha < 1$. Hence, the integer $c$ can be chosen 
so large that $2 L c^{\alpha - 1} \leq  1$. Therefore,
\begin{eqnarray*}
f(cn,P\otimes  R^{k_1,...,k_d},d)
\le (L c^{d-1})(2c^d g(n))+c^dg(n)
\le [2L  c^{\alpha(k_1, \ldots, k_d) - 1}]c^{d}g(cn)+c^dg(cn)
\le2c^dg(cn) \ ,
\end{eqnarray*}
where we use $g(n) = K n^{d - \alpha}$ in the second inequality.  This completes our induction and hence proves equation (\ref{cm}).

Finally,  we estimate $f(n,P\otimes R^{k_1, \ldots , k_d}, d)$ for all positive integers $n$.
\begin{eqnarray*}
f(n, P \otimes R^{k_1,...,k_d}, d) &=& f(c^{\log_c n},P \otimes R^{k_1, \ldots , k_d},d) \\
&\le& f(c^{\lceil\log_c n \rceil},P \otimes R^{k_1, \ldots , k_d},d) \\
& \le & 2c^dg(c^{\lceil\log_c n \rceil}) \\
&\le& 2c^dg(c^{\log_c n +1}) \\
&=& 2c^dg(cn) \\
&\le& 2c^d c^dg(n),
\end{eqnarray*}
where $\lceil\log_c n \rceil$ is the smallest integer $\ge \log_c n$, 
and we use (\ref{cm}) in the second inequality and $g(n) = K n^{d - \alpha}$ in the last inequality.
This proves the right inequality of (\ref{block}).

The proof of Theorem \ref{super} is completed.
\end{proof}

We conclude this section with an observation. If only one of $k_1, \ldots , k_d$ is greater than one, the matrix $P \otimes R^{k_1, \ldots, k_d}$
 is a tuple permutation matrix. By Proposition \ref{alpha}, $\alpha(k_1, \ldots , k_d) = 1$. The proof of Theorem \ref{super} fails in this case, but it can be 
modified to show that $f(n, P \otimes R^{k_1, \ldots , k_d}, d) = O(n^{d - 1 + \epsilon})$, where $\epsilon$ is an arbitrarily small positive number.  To see this, we can replace $g(n)$ of (\ref{Kconstant}) by $g(n) = K n^{d - 1 + \epsilon}$ and choose $c$ so large that $2 L c^{- \epsilon} \leq  1$. In the next section, we improve this result and show that $f(n, P \otimes R^{k_1, \ldots , k_d}, d) = O(n^{d - 1})$. The  method is quite different from that of this section.

\section{Tuple permutation matrices}
\label{tuple}
In this section, we study the extremal function of an arbitrary tuple permutation matrix. As previously mentioned, a tuple permutation matrix is the Kronecker product of a $d$-dimensional permutation matrix and  $ R^{k_1, \ldots , k_d}$, where only one of $k_1, \ldots, k_d$ is larger than unity. We improve Geneson's ideas for $d=2$ case \cite{G} and obtain a tight bound on the extremal function for $d \geq 2$.
 
Suppose $P$ is a permutation matrix. We call a matrix $P \otimes R^{k_1, \ldots , k_d}$ a $j$-tuple  permutation matrix generated by $P$ if one of $k_1, \ldots , k_d$ is equal to $j$ and the rest are unity. In particular, a $j$-tuple permutation matrix is called a double permutation matrix if $j=2$.

Let
\begin{displaymath}
F(n,j,k,d)=\max_{M} f(n,M,d) \ ,
\end{displaymath}
where $M$ ranges through all $d$-dimensional $j$-tuple permutations matrices generated by $d$-dimensional $k \times \cdots \times k$ permutation matrices.

\begin{theo}
\label{main}
For all $j \ge 2$,
$F(n,j,k,d)=\Theta(n^{d-1})$.
\end{theo}

\noindent
The proof of this theorem is based on a series of lemmas.

Since $F(n,j,k,d)$ has $n^{d-1}$ as a  lower bound in view of Proposition \ref{Easy}, it suffices to prove that it has upper bound $O(n^{d-1})$.

We first observe that $F(n,j,k,d)$ and $F(n,2,k,d)$ are bounded by each other.
\begin{lem}
\label{2j}
$F(n, 2, k, d) \leq F(n, j, k, d) \leq (j-1) F(n, 2, k, d) ~~~~~ \mbox{for $j > 2$} \ .$
\end{lem}
\begin{proof}
It suffices to show that 
\begin{equation}
\label{2j'}
f(n, P, d) \leq f(n, P', d) \leq (j-1) f(n, P, d) ,
\end{equation}
where $P$ is a double permutation $2k \times k \times \cdots \times k$ matrix, $P'$ is a $j$-tuple permutation $jk \times k \times \cdots \times k$ matrix, and  both $P$ and $P'$ are generated from the same arbitrary permutation matrix of size $k \times \cdots \times k$.

The left inequality of (\ref{2j'}) follows from the fact that a $d$-dimensional $n \times \cdots \times n$ matrix that avoids $P$ must also avoid $P'$.

To prove the right inequality, we suppose $A$ is a $d$-dimensional $n \times \cdots \times n$ matrix that avoids $P'$ and has $f(n, P', d)$ nonzero entries. 
In each $1$-row of $A$, we list all the $1$-entries $e_1, e_2, \ldots$ in the order of increasing first coordinates and then change all the $1$-entries in this $1$-row 
except $e_1, e_{j}, e_{2j-1}, \ldots$ to $0$-entries. In this way, we obtain a new matrix $A'$, which avoids $P$ since $A$ avoids $P'$.
This together with $|A| \leq (j-1)|A'|$, where $|M|$ denotes the number of $1$-entries in $M$, justifies the right inequality of (\ref{2j'}).
\end{proof}

In view of Lemma \ref{2j}, it suffices to study the upper bound on $f(n, P, d)$, where $P$ is a $d$-dimensional double permutation matrix of size 
$2k \times k \times \cdots \times k$.

Suppose $A$ is an arbitrary $d$-dimensional $kn \times \cdots \times kn$ matrix that avoids $P$. As in Section 2, we study the $S$ submatrices of $A$, which are constructed by dividing $A$ into $n^d$ disjoint submatrices of size $k \times \cdots \times k$ and labeling these submatrices as $S(i_1, \ldots , i_d)$.

The contraction matrix of $A$ is defined to be the $d$-dimensional $n \times \cdots \times n$ matrix $C = \left(c_{i_1,i_2, \ldots , i_d}\right)$ such that $c_{i_1,i_2, \ldots ,i_d}=1$ if $S(i_1, i_2, \ldots , i_d)$ is a nonzero matrix and $c_{i_1,i_2, \ldots , i_d}=0$ if $S(i_1, i_2, \ldots ,i_d)$ is a zero matrix.

We now construct a  $d$-dimensional $n \times \cdots \times n$ zero-one matrix $Q=(q_{i_1, \ldots , i_d})$.  Each entry $q_{i_1, \ldots , i_d}$ is defined based on the $S$ submatrices of $A$.
\begin{enumerate}
\item  $q_{i_1, \ldots , i_d}=0$  if $S(i_1, \ldots , i_d)$ is a zero matrix.

\item  $q_{i_1, \ldots , i_d}=1$ if $S(i_1, i_2, \ldots , i_d)$ is a nonzero matrix and $S(1,i_2, \ldots , i_d)$, $\ldots$, $S(i_1-1,i_2, \ldots , i_d)$ are all zero matrices.

\item Let $x$ be the largest integer less than $i_1$ for which $q_{x,i_2, \ldots , i_d}=1$. Then define $q_{i_1, i_2, \ldots , i_d}=1$ if the augmented matrix 
formed by submatrices $S(x, i_2, \ldots , i_d)$, $\ldots$ , $S(i_1, i_2, \ldots , i_d)$ contains at least two $1$-entries in the same $1$-row, and  $q_{i_1, \ldots , i_d}= 0$ otherwise.

\end{enumerate}

\noindent
\begin{lem}
\label{Q}
$Q$ avoids $P$.
\end{lem}
\begin{proof}
Suppose to the contrary that $Q$ contains $P$. Suppose the $1$-entries $e_1$, $e_2$, $\ldots$ , $e_{2k}$, where $e_{2i-1}$ and  $e_{2i}$ are in the same $1$-row, form a copy of $P$ in $Q$. Denote $e_{2i-1}= q_{x_1, x_2, \ldots , x_d}$ and $e_{2i}= q_{x_1',x_2, \ldots, x_d}$, where $x_1 < x_1'$. Then, by the definition of matrix $Q$, the augmented matrix formed by $S(x_1, x_2, \ldots , x_d), \ldots , S(x_1', x_2, \ldots , x_d)$ contains two $1$-entries, denoted by $f_{2i-1}$ and $f_{2i}$, in the same $1$-row of $A$. The one-entries
$f_1, \ldots , f_{2k}$ form a copy of $P$ in $A$, a contradiction.
\end{proof}

We now study those $S$ submatrices of $A$ which contain two nonzero entries in the same $1$-row. The next lemma is the key difference between our approach and Geneson's approach \cite{G} even for $d=2$.
\begin{lem}
\label{wide}
$A$ has at most $F(n, 1, k, d)$ total $S$ submatrices with two nonzero entries in the same $1$-row.
\end{lem}
\begin{proof} We assume to the contrary that $A$ has more than $F(n,1,k,d)$ such $S$ submatrices. Let $A'$ be formed by changing all $1$-entries in all other $S$ submatrices to $0$-entries in $A$. Suppose that the double permutation matrix $P$ is generated from the permutation matrix $P'$ and that $C'$ is the contraction matrix of $A'$. Matrix $C'$ has more than $F(n,1,k,d) \ge f(n,P',d)$ $1$-entries, so it must contain $P'$. Denote by $e_1, \ldots , e_k$  the $1$-entries in $C'$ forming a copy of $P'$.  Then each of $S(e_1), \ldots, S(e_k)$ is a $S$ submatrix of $A'$ that has at least two nonzero entries in the same $1$-row. All of these pairs of nonzero entries in $S(e_1), \ldots , S(e_k)$ form a copy of $P$ in $A'$.  Hence, $A'$ contains $P$ and so does $A$, a contradiction.
\end{proof}

For each 1-entry $q_{i_1, i_2, \ldots  , i_d}=1$ of $Q$, we define a chunk $C^*(i_1, i_2, \ldots , i_d)$, which is an augmented matrix formed by consecutive $S$ submatrices, as follows \cite{G}. 

\begin{enumerate}

\item If $q_{i_1, i_2, \ldots , i_d}=1$ and $i_1'$ is the smallest integer greater than $i_1$ such that $q_{i_1', i_2, \ldots , i_d}=1$,  then the chunk $C^*(i_1, i_2, \ldots , i_d)$ is defined to be the augmented matrix formed by $S(i_1, i_2, \ldots , i_d)$, $\ldots$ , $S(i_1' - 1, i_2, \ldots , i_d)$.

\item If $q_{i_1, i_2, \ldots , i_d}=1$ and there is no $i_1'>i_1$ such that $q_{i_1', i_2, \ldots , i_d}=1$, then $C^*(i_1, i_2, \ldots  , i_d)$ is the augmented matrix formed by $S(i_1, i_2, \ldots , i_d)$, $\ldots$ , $S(n, i_2, \ldots , i_d)$.

\end{enumerate}

We call a chunk {\it $j$-tall}, where $j=2,3, \ldots , d$, if each of its $j$-cross sections contains at least one 1-entry. The ($d-1$)-dimensional matrix $M' = (m'_{i_1, \ldots, i_{j-1}, i_{j+1}, \ldots, i_d})$ is called the $j$-remainder of a $d$-dimensional matrix $M = (m_{i_1, \ldots , i_d})$ if $m'_{i_1, \ldots , i_{j-1},i_{j+1}, \ldots , i_d}$ is defined to be $1$
when there exists $i_j$ such that $m_{i_1, \ldots , i_d}=1$ and to be $0$ otherwise.

\begin{lem}
\label{tall}
For each $j= 2, 3, \ldots, d$ and each $m=1, \ldots , n$, $A$ has at most $F(n,1+k^{d-2},k,d-1)$ total $j$-tall chunks of the form $C^*(i_1, \ldots , i_{j-1},m,i_{j+1}, \ldots , i_d)$.
\end{lem}
\begin{proof}
Assume to the contrary that $A$ has $r$ chunks $C^*_1, C^*_2, \ldots , C^*_r$, where $r>F(n,1+k^{d-2},k,d-1)$,  of the form $C^*(i_1, \ldots , i_{j-1},m,i_{j+1}, \ldots , i_d)$ 
that have $1$-entries in
all their $j$-cross sections. Let $S_1, S_2, \ldots, S_r$ be the starting $S$ submatrices of the chunks $C^*_1, C^*_2, \cdots , C^*_r$, respectively. Let $A'$ be the matrix formed
by changing all $1$-entries of $A$ that do not lie in the chunks $C^*_1, \ldots , C^*_r$ to $0$-entries. We further change all the $1$-entries of $A'$ that do not sit in 
$S_1, \ldots, S_r$ to $0$-entries and denote the resulting matrix by $A''$. 
Denote by $C$ the contraction matrix of the $j$-remainder of $A''$. Then $C$ is a $(d-1)$-dimensional $n \times \cdots \times n$ matrix and it has $r$ ones so it contains every ($1+k^{d-2}$)-tuple 
($d-1$)-dimensional permutation matrix.

We now pick a $(d-1)$-dimensional $(1 + k^{d-2})$-tuple permutation matrix. Since $P$ is a $d$-dimensional double permutation matrix of size $2k \times k \times \cdots \times k$ and $j \neq 1$, the $j$-remainder of $P$ is a $(d-1)$-dimensional double permutation matrix
of size $2k \times k \times \cdots \times k$. We denote by $P'$ the $(d-1)$-dimensional $(1 + k^{d-2})$-tuple permutation matrix of size $(1+k^{d-2})k \times k \times \cdots \times k$ such that $P'$ and the $j$-remainder of $P$ are generated from the same $(d-1)$-dimensional permutation matrix.

For each pair of ones in a row of $P$ with coordinates $(x_1,x_2, \ldots , x_d)$ and $(x_1 + 1,x_2, \ldots , x_d)$, $P'$ has corresponding  ($1 + k^{d-2}$) 
ones with coordinates $(\tilde{x}_1, x_2, \ldots, x_{j-1}, x_{j+1}, \ldots , x_d)$,   $(\tilde{x}_1 + 1, x_2, \ldots , x_{j-1}, x_{j+1}, \ldots , x_d)$,  $\cdots$ , 
$(\tilde{x}_1 + k^{d-2}, x_2, \ldots , x_{j-1}, x_{j+1}, \ldots, x_d)$  
in a single $1$-row. Since $C$ contains $P'$, this set of ($1  + k^{d-2}$) ones is represented by $1$-entries with coordinates 
$(t_1(\lambda), t_2, \ldots , t_{j-1}, t_{j+1}, \ldots , t_d)$, where $\lambda = 1, 2, \ldots , 1 + k^{d-2}$,  in the same $1$-row of $C$.

Let $S(t_1(\lambda), t_2, \ldots , t_{j-1}, m, t_{j+1}, \ldots , t_d)$,   $1 \leq \lambda \leq 1 + k^{d-2}$, be the corresponding $S$ submatrices of $A'$. By the construction of $A'$, $A''$ and $C$, these $S$ submatrices are the starting $S$ submatrices of some of the chunks $C^*_1, \ldots , C^*_r$. Each of these ($1 + k^{d-2}$) chunks has $1$-entries in every $j$-cross section; in particular each chunk has a nonzero entry 
with the same $j^{\text{th}}$ coordinate $(m-1)k+x_j$. There are at least $1+k^{d-2}$ nonzero entries with this given $j^{\text{th}}$ coordinate in these chunks, but there are $k^{d-2}$ $1$-rows in a $j$-cross section of these chunks. By the pigeonhole principle,  there exist a pair of $1$-entries in the same $1$-row of $A'$.

Hence, for each pair of ones in the same $1$-row of $P$, we have a corresponding pair of ones in the same $1$-row of $A'$. Since two $1$-entries of $P$ not in the same 
$1$-row differ in all their coordinates, $A'$ contains $P$, and so does $A$; a contradiction.
\end{proof}

We can now derive a recursive inequality on $F(n, j, k, d)$, the resolution of which gives an upper bound on $F(n,j,k,d)$.

\begin{lem}
\label{Ine}
Let $d$, $s$, $n$ be positive integers where $d \ge 2$. Then
\begin{eqnarray}
F(kn,2,k,d) &\le& (d-1)nk^{d-1}F(n, 1+k^{d-2},k,d-1) + k^dF(n,1,k,d) 
 +(k-1)^{d-1}F(n,2,k,d) . ~~~~~
\label{IN}
\end{eqnarray}
\end{lem}
\begin{proof} We count the maximum number of $1$-entries in $A$ by counting the number of ones in three types of chunks of $A$.

\medskip

\noindent
{\bf Case 1: chunk has two $1$-entries in the same $1$-row} 

In view of the definitions of matrix $Q$ and a chunk, such a chunk has only one nonzero $S$ submatrix so it has at most $k^d$  nonzero entries. By Lemma \ref{wide}, there are at most $F(n,1,k,d)$ such $S$ submatrices. Chunks of this type contain at most $k^d F(n, 1, k, d)$ nonzero entries.

\medskip

\noindent
{\bf Case 2: chunk is $j$-tall for some $j=2, 3, \ldots, d$ and has no two $1$-entries in the same $1$-row}

There are $(d-1)$ choices for $j$-tall since $j=2, 3, \ldots , d$. For each $j$, the integer $m$ of Lemma \ref{tall} can be $1, \ldots , n$. A $j$-tall chunk with no two $1$-entries in the same row has at most $k^{d-1}$ $1$-entries. For each pair of $j$ and $m$,  there are at most $F(n, 1 + k^{d-2}, k, d-1)$ such chunks in view of Lemma \ref{tall}. In total, chunks of this type contain at most $(d-1)nk^{d-1}F(n, 1+k^{d-2},k,d-1)$ nonzero entries.

\medskip

\noindent
{\bf Case 3: chunk is not $j$-tall for any $j=2, 3, \ldots , d$ and has no two $1$-entries in the same $1$-row}

Such a chunk has at most $(k-1)^{d-1}$ ones. By the definition of a chunk, the number of chunks is equal to the number of nonzero entries in matrix $Q$, which, by Lemma \ref{Q}, has 
at most $F(n,2,k,d)$ nonzero entries. There are at most $(k-1)^{d-1} F(n,2,k,d)$ ones in chunks of this type.

Summing all cases proves Lemma \ref{Ine}.
\end{proof}

We are now ready to finish the proof of  Theorem \ref{main}.

\begin{proof}[Proof of Theorem \ref{main}]
We proceed by induction on $d$. The base case of $d=1$ is trivial. We then make the inductive assumption that 
\begin{equation}
F(n,j,k,d-1)=O(n^{d-2}) ~~ \mbox{for some $d \geq 2$}  \label{inductive}
\end{equation}
and prove that $F(n,j,k,d) = O(n^{d-1})$.

We first use Lemma \ref{Ine} to show that 
\begin{equation}
F(n,2,k,d) \leq k(c+dk)n^{d-1} \label{j=2} ,
\end{equation}
where $c$ is a positive constant to be determined.

We simplify inequality (\ref{IN}) of Lemma \ref{Ine}. Inductive assumption (\ref{inductive}) implies that $F(n, 1+ k^{d-2},k,d-1)=O(n^{d-2})$.  
We also have $F(n,1,k,d)=O(n^{d-1})$, which was proven by Marcus and Tardos \cite{MT} for $d=2$ and by  Klazar and Marcus \cite{KM} for $d > 2$. Hence,  we can choose a sufficiently large constant $c$ 
such that the sum of the first two terms on the right hand side of (\ref{IN}) is bounded by $c n^{d-1}$. Therefore,
\begin{equation}
 F(kn,2,k,d) \le (k-1)^{d-1}F(n,2,k,d)+cn^{d-1}  ~~~~~ \mbox{for all $n$.} \label{sn}
\end{equation}

We then use another induction, which is a strong induction on $n$, to prove inequality (\ref{j=2}). The base case of $n \leq k$ is trivial.  Assuming that (\ref{j=2}) is true for all $n < m$, we show that (\ref{j=2}) also holds for $n =m$.

Let $N$ be the maximum integer that is less than $m$ and divisible by $k$. A $d$-dimensional $m \times \cdots \times m$ zero-one  matrix has at most $m^d-N^d \le m^d-(m-k)^d \le dkm ^{d-1}$ more entries than a $d$-dimensional $N \times N \times \cdots \times N$ matrix. Thus we have $F(m,2,k,d)\le F(N,2,k,d)+dkm^{d-1}$. This together with (\ref{sn}) gives
\begin{eqnarray*}
F(m, 2, k, d) &\le&  (k-1)^{d-1}F\left({N \over k}, 2,k,d\right)+c\left({N \over k}\right)^{d-1}+dkm^{d-1} \\
&\le&  (k-1)^{d-1}k(c+dk) \left({N \over k}\right)^{d-1} + c\left({N \over k}\right)^{d-1}+dkm^{d-1} \\
&\le& (k-1)(c+dk)N^{d-1}+(c+dk)m^{d-1} \\
&\le& k(c+dk)m^{d-1} \ ,
\end{eqnarray*}
where we use the strong inductive assumption in the second inequality. Hence, inequality (\ref{j=2}) holds for $n=m$. The strong induction shows that (\ref{j=2}) is true for all positive integers $n$.

Having verified the inequality (\ref{j=2}), we continue to complete the induction on $d$ by showing that $F(n,j,k,d)=O(n^{d-1})$. This easily follows from inequality (\ref{j=2}) and Lemma \ref{2j}.
We have completed the induction.

Since $F(n,j,k,d) = \Omega(n^{d-1})$ in view of Proposition \ref{Easy}, this together with  $F(n,j,k,d)=O(n^{d-1})$ completes the proof of Theorem \ref{main}.
\end{proof}

We conclude this section with a remark. In the paragraph between two inequalities (\ref{j=2}) and (\ref{sn}), we use Klazar and Marcus' result \cite{KM} $F(n, 1, k, d) = O(n^{d-1})$ to 
choose the constant $c$ in (\ref{j=2}). In fact, Klazar and Marcus gave a more refined upper bound ${F(n,1,k, d) \over n^{d-1}} = 2^{O(k \log k)}$. This allows us to improve the inductive assumption (\ref{inductive}) to ${F(n,j,k,d-1) \over n^{d-2}} = 2^{O( k \log k)}$ and choose 
$c = 2^{O(k \log k)}$.  In this way, we are able to prove ${F(n,j,k,d) \over n^{d-1}} = 2^{O(k \log k)}$. 

In the next section, we improve Klazar and Marcus upper bound from $2^{O(k \log k)}$ to $2^{O(k)}$. As a consequence, $c = 2^{O(k)}$ and hence ${F(n,j,k,d) \over n^{d-1}} = 2^{O(k)}$. Lemma \ref{wide} is crucial in making the extension from ${F(n,1,k,d) \over n^{d-1}} = 2^{O(k)}$ to ${F(n,j,k,d) \over n^{d-1}} = 2^{O(k)}$ possible.

\section{Limit inferior and limit superior }
\label{constant}
In this section, we consider matrices $P$ such that $f(n, P, d) = \Theta(n^{d-1})$. This tight bound implies that $\{{ f(n, P, d) \over n^{d - 1}} \}$ is a bounded sequence. We are interested in the limits of this sequence.

When $d=2$, Pach and Tardos showed that $f(n, P, 2)$ is super-additive \cite{PT}.
By Fekete's Lemma on super-additive sequences \cite{Fekete}, the sequence $\{ {f(n, P, 2) \over n} \}$ is convergent. The limit is known as the F\"uredi-Hajnal limit.

When $d > 2$,  it is still an open problem to prove the convergence of the sequence $\{ {f(n, P, d) \over n^{d -1 }} \}$. Instead, we consider the limit inferior and limit superior of the sequence and define
\begin{displaymath}
\label{ISd}
I(P, d) = \liminf_{n \rightarrow \infty} { f(n, P, d) \over n^{d - 1}} \ , ~~~~~   S(P, d) = \limsup_{n \rightarrow \infty} { f(n, P, d) \over n^{d - 1}} \ .
\end{displaymath}
We derive lower bounds on $I(P,d)$ and an upper bound on $S(P,d)$. These bounds are written in terms of the size of $P$.

The main ideas in this section are Fox's interval minor containment \cite{Fox} and our observation that the extremal function is super-homogeneous in higher dimensions.

\subsection{An improved upper bound}

Klazar and Marcus \cite{KM} showed that $S(P, d) =2^{O(k \log k)}$ 
for $k \times \cdots \times k$ permutation matrices $P$. In this subsection, we extend Fox's ideas for the $d=2$ case \cite{Fox} to improve this upper bound to $2^{O(k)}$ for $d \geq 2$. We then show that the new upper bound also holds for tuple permutation matrices, which is a new result even for $d=2$.

\begin{theo} 
\label{IS}
If $P$ is a $d$-dimensional $k \times \cdots \times k$ permutation matrix or a tuple permutation matrix generated by such a permutation matrix, then
$S(P,d) = 2^{O(k)}$.
\end{theo}

The proof uses the notion of cross section contraction and interval minor containment \cite{Fox}. Contracting several consecutive $\ell$-cross sections of a $d$-dimensional matrix means that we replace these $\ell$-cross sections by a single $\ell$-cross section, placing a one in an entry of the new cross section if at least one of the corresponding entries in the original $\ell$-cross sections is a $1$-entry and otherwise placing a zero in that entry of the new cross section. The contraction matrix, as defined in Section 3, of an $sn \times \cdots \times sn$ matrix $A$ can be obtained by contracting every $s$ consecutive $\ell$-cross sections of $A$ uniformly for $1 \leq \ell \leq d$.

We say that $A$ contains $B$ as an interval minor if we can use repeated cross section contraction to transform $A$ into a matrix which contains $B$. Matrix $A$ avoids $B$ as an interval minor if $A$ does not contain $B$ as an interval minor. 

Equivalently, a $k_1 \times k_2 \times \cdots \times k_d$ matrix $B =  (b_{i_1, i_2, \ldots , i_d} )$ is an interval minor of a matrix $A$ if
\begin{itemize}

\item for each $i=1, \ldots , d$, there are $k_i$ disjoint intervals, $W_{i, 1}, \ldots , W_{i, k_i}$, which are sets of  consecutive positive integers,

\item and if $b_{i_1, \ldots , i_d} = 1$ then the submatrix $W_{1, i_1} \times \cdots \times W_{d, i_d}$ of $A$ contains a $1$-entry.

\end{itemize}

The containment in previous sections is generally stronger than containment as an interval minor. Indeed, $A$ contains $B$ implies that $A$ contains $B$ as an interval minor. However, since a permutation matrix has only one $1$-entry in every cross section, containment of a permutation matrix $P$ is equivalent to containment of $P$ as an interval minor.

Analogous to $f(n, P, d)$, we define $m(n, P, d)$ to be the maximum number of $1$-entries in a $d$-dimensional $n \times \cdots \times n$ zero-one matrix that avoids $P$ as an interval minor.

We observe that
\begin{equation}
\label{fm}
f(n, P, d) \leq m(n, R^{k, \ldots, k}, d) 
\end{equation}
for every $k \times \cdots \times k$ permutation matrix $P$. This follows from the fact that containment of $R^{k, \ldots, k}$ as an interval minor implies containment of $P$. Hence, we seek an upper bound on $m(n, R^{k, \ldots, k}, d)$. We denote by $f_{k_1, \ldots , k_d}(n,t,s, d)$ the maximum number of $1$-rows that have at least $s$ nonzero entries in a $d$-dimensional $t \times n \times\cdots \times n$ matrix that avoids 
$R^{k_1, \ldots , k_d}$ as an interval minor.

\begin{lem}
\label{mblock}

\begin{equation}
m(tn,R^{k, \ldots, k},d) \le s^d m(n,R^{k, \ldots, k},d)+dn t^df_{k, \ldots, k}(n,t,s, d) ,
\end{equation}

\end{lem}
\begin{proof}
Let $A$ be a $d$-dimensional $tn \times \cdots \times tn$ matrix that avoids  $R^{k, \ldots , k}$ as an interval minor and has $m(tn,R^{k, \ldots , k},d)$ $1$-entries. Partition $A$ uniformly into $S$ submatrices of size $t \times \cdots \times t$. Let $C$ be the contraction matrix of $A$ as defined in Section 3.

We do casework based on whether an $S$ submatrix of $A$ has $s$ nonzero $\ell$-cross sections for some $\ell$.

We first count the number of $1$-entries from the $S$ submatrices which do not have $s$ nonzero $\ell$-cross sections for any $\ell$. The contraction matrix $C$ has at most
$m(n,R^{k, \ldots, k},d)$ $1$-entries for, otherwise, $C$ contains $R^{k, \ldots, k}$ as an interval minor, and thus $A$ contains $R^{k, \ldots, k}$ as an interval minor as well, 
a contradiction.
Hence,  $A$ has at most $m(n,R^{k, \ldots, k},d)$ such $S$ submatrices, each of which contains at most $(s-1)^d <s^d$ $1$-entries. There are at most $s^dm(n,R^{k, \ldots, k},d)$ $1$-entries from the $S$ submatrices of this type.

We next count the number of $1$-entries from the $S$ submatrices that have $s$ nonzero $\ell$-cross sections for some $\ell$. Let $A'$ be the matrix obtained from $A$ by changing all 
the $1$-entries from the $S$ submatrices of $A$ that do not have $s$ nonzero $\ell$-cross sections to $0$-entries. Without loss of generality we let $\ell=1$.  Divide $A'$ into $n$ blocks, each of which is a $t \times tn \times \cdots \times tn$ submatrix of $A'$. For each block, contract every $t$ consecutive $j$-cross sections uniformly for 
all $j \not =1$ to get  a $t \times n \times \cdots \times n$ matrix, which has at most $f_{k, \ldots , k}(n,t,s, d)$ $1$-rows with at least $s$ nonzero entries in each $1$-row.  
Hence, in each block, there are at most $f_{k, \ldots , k}(n,t,s,d )$  nonzero $S$ submatrices. As we range over 
all $n$ blocks and as $\ell$ ranges from $1$ to $d$, we have at most $dnt^df_{k, \ldots , k}(n,t,s, d)$ ones from $S$ submatrices of this type.

Summing both cases proves Lemma \ref{mblock}.
\end{proof}

% mention the similarity of this to Nivasch's?

It remains to find an upper bound on $f_{k, \ldots , k}(n,t,s, d)$. We prove the following recursive inequality.
\begin{lem}
\label{f}
$f_{k_1, \ldots , k_d}(n,2t,2s, d) \le 2 f_{k_1, \ldots , k_d}(n,t,2s, d)+2f_{k_1 - 1, k_2, \ldots , k_d}(n,t,s, d)$.
\end{lem}
\begin{proof}
Let $A$ be a $d$-dimensional $2t \times n \times \cdots \times n$ matrix that avoids $R^{k_1, \cdots, k_d}$ as an interval minor and has $f_{k_1, \ldots , k_d}(n,2t,2s, d)$ $1$-rows, each of which has at least $2s$ ones.

The first type of these $1$-rows have all their $1$-entries among their first $t$ or last $t$ entries. There are clearly at most $2 f_{k_1, \ldots , k_d}(n,t,2s, d)$ such $1$-rows in $A$.

The other type of these $1$-rows must have at least one $1$-entry among  both the first $t$ and the last $t$ entries. Since each $1$-row in question has at least $2s$ ones, 
there are at least $s$ ones among either the first or last $t$ entries. Without loss of generality, we consider those $1$-rows in which the first $t$ entries contain at least $s$ $1$-entries. 
Let $A'$ be the matrix obtained from $A$ by changing all $1$-entries to $0$-entries in all other $1$-rows and then contracting the last $t$ $1$-cross sections. Hence, the last entry 
in each nonzero $1$-row of $A'$ is a $1$-entry.  The first $t$ $1$-cross sections of $A'$ must avoid $R^{k_1-1, \ldots , k_d}$ as an interval minor for, otherwise, $A'$ 
contains $R^{k_1, \ldots , k_d}$ as an interval minor and so does $A$, a contradiction. Thus, there are at most $2f_{k_1-1, k_2, \ldots , k_d}(n,t,s, d)$ $1$-rows in which 
both the first $t$ and last $t$ entries include at least one $1$-entry.

Adding up both cases gives the result.
\end{proof}

The recursive inequality of Lemma \ref{f} allows us to get an upper bound on $f_{l, k, \ldots , k}(n,t,s,d)$. 
\begin{lem}
\label{closed}
If $s$, $t$ are powers of $2$ and $2^{\ell -1} \le s \le t$, then 
$f_{\ell ,k, \ldots, k}(n,t,s, d) \leq {2^{\ell -1}t^2 \over s} \ m(n, R^{k, \ldots, k}, d-1)$.
\end{lem}
\begin{proof}
We induct on $\ell$. For $\ell =1$, we show that $f_{1,k, \ldots, k}(n,t,s, d) \le m(n,R^{k, \ldots, k},d-1)$. Suppose on the
contrary $f_{1,k, \ldots, k}(n,t,s, d) > m(n,R^{k, \ldots, k},d-1)$. Then there is a $t \times n \times \cdots \times n$ matrix $A$ which avoids $R^{1,k,\ldots,k}$
as an interval minor and has more than $m(n,R^{k, \ldots, k},d-1)$ $1$-rows with at least $s$ $1$-entries in each $1$-row. Let $B$ be
the $1 \times n \times \cdots \times n$ matrix obtained from $A$ by contracting all the $1$-cross sections. Then $B$, which can be viewed as a $(d-1)$-dimensional matrix, has over $m(n,R^{k, \ldots, k},d-1)$ $1$-entries 
and thus contains the $(d-1)$-dimensional matrix $R^{k, \ldots, k}$ as an interval minor. Consequently, $A$ contains the $d$-dimensional $R^{1, k, \ldots, k}$ as
an interval minor, a contradiction. Therefore, $$f_{1,k, \ldots, k}(n,t,s, d) \le m(n,R^{k, \ldots, k},d-1)  \leq {2^{1-1}t^2 \over s} \ m(n, R^{k, \ldots, k}, d - 1) \ , $$ 
which proves the base case.

Assuming that for all $s$ and $t$ that are powers of $2$ satisfying $2^{\ell-2} \le s \le t$ we have 
\begin{equation}
\label{f_i}
f_{\ell-1,k, \ldots, k}(n,t,s, d) \leq {2^{\ell-2}t^2 \over s}\ m(n, R^{k, \ldots, k}, d - 1)
\end{equation}
 for some $\ell \geq 2$, we need to show that 
\begin{equation}
\label{f_}
f_{\ell ,k, \ldots, k}(n,t,s, d) \leq {2^{\ell -1}t^2 \over s} \ m(n, R^{k, \ldots, k}, d - 1) 
\end{equation}
for all $s$ and $t$ that are powers of $2$ satisfying $2^{\ell -1} \le s \le t$. 

We use another induction on $t$ to show that (\ref{f_}) is true for all $t \geq s$ that are powers of $2$. The base case of $t=s$ is trivial.  If $f_{\ell,k, \ldots, k}(n,t,s, d) 
\leq {2^{\ell-1}t^2 \over s} m(n, R^{k, \ldots, k}, d - 1) $ for some $t \geq s$ that is a power of $2$, we prove the same inequality for $2t$.
By Lemma \ref{f}, we have 
\begin{eqnarray*}
f_{\ell,k, \ldots, k}(n,2t,s, d) &\le& 2f_{\ell,k, \ldots, k}(n,t,s, d)+2f_{\ell-1,k, \ldots, k}(n,t,s/2, d) \\
&\le & 2{2^{\ell-1}t^2 \over s} m(n, R^{k, \ldots, k}, d - 1) +2{2^{\ell-2}t^2 \over s/2} m(n, R^{k, \ldots, k}, d - 1)  \\ 
&=& {2^{\ell-1}(2t)^2 \over s} m(n, R^{k, \ldots, k}, d - 1) \ ,
\end{eqnarray*}
where we use the two inductive assumptions in the second inequality. 
Thus our induction on $t$ is complete and (\ref{f_}) is proved. As a result, our induction on $l$ is also complete.
\end{proof}

We are now ready to prove Theorem \ref{IS}.
\begin{proof}[Proof of Theorem \ref{IS}]
We first bound the right hand side of inequality (\ref{fm}). 
We claim that
\begin{equation}
\label{mO}
{m(n,R^{k, \ldots, k},d) \over n^{d -1}} =2^{O(k)} \ .
\end{equation}

The base case of $d=1$ is trivial. Assuming that (\ref{mO}) is true for $(d-1)$, we
combine Lemmas \ref{mblock} and \ref{closed} to get
$$m(tn,R^{k, \ldots , k},d) \le s^dm(n,R^{k, \ldots, k}, d)+dt^d{2^{k-1}t^2 \over s}2^{O(k)}n^{d-1} .$$
Choosing $t=2^{dk }$ and $s=2^{k-1}$ yields
$$m(2^{dk}n,R^{k, \ldots , k},d) \le 2^{(k-1)d}m(n,R^{k, \ldots, k},d)+d  2^{kd(d + 2)} 2^{O(k)}n^{d-1} .$$
In particular, if $n$ is a positive integer power of $2^{dk}$, iterating this inequality yields
\begin{eqnarray*}
\lefteqn{m((2^{dk})^L, R^{k, \ldots, k}, d)} \\
& \leq& 2^{(k-1)d} m((2^{dk})^{L - 1}, R^{k, \ldots, k}, d)  
 + d 2^{kd(d+2)} 2^{O(k)} (2^{dk})^{(L-1)(d-1)} \\
&\leq & 2^{2(k-1)d}  m((2^{dk})^{L - 2}, R^{k, \ldots, k}, d)  
 + \  d 2^{kd(d+2)} 2^{O(k)} \left( 1 + {1 \over 2^{d(dk - 2k +1)}} \right) (2^{dk})^{(L-1)(d-1)} \\
&\leq& 2^{L(k-1)d}  m(1, R^{k, \ldots, k}, d)  
 + \ d 2^{kd(d+2)} 2^{O(k)} \left(  1 + 
{1 \over 2^{d(dk - 2k +1)}} + {1 \over 2^{2d(dk - 2k +1)}} + \cdots \right) (2^{dk})^{(L-1)(d-1)} \\
&=& 2^{O(k)} (2^{dk})^{(L-1)(d-1)} \ .
\end{eqnarray*}
Hence, if $(2^{kd})^{L-1} \leq n < (2^{kd})^L$, then
$$m(n, R^{k, \ldots, k}, d) \leq m((2^{dk})^L, R^{k, \ldots, k}, d) = 2^{O(k)} (2^{dk})^{(L-1)(d-1)} \leq 2^{O(k)} n^{d-1} \ .$$
This completes the induction on $d$, and hence (\ref{mO}) is proved.

It follows from (\ref{fm}) and (\ref{mO}) that Theorem \ref{IS} is true for every permutation matrix $P$. By the remark at the end of Section 3, this result can be extended to tuple permutation matrices. The proof of Theorem \ref{IS} is completed.
\end{proof}

\subsection{
Lower bounds and super-homogeneity}

We first use Cibulka's method in \cite{C} to show that $I(P, d) \geq d(k-1)$ for all permutation matrices of size $k \times \cdots \times k$ and extend this  lower bound to tuple permutation matrices. 
\begin{theo}
\label{llb}
If $P$ is a $d$-dimensional $k \times \cdots \times k$ permutation matrix or a tuple permutation matrix generated by such a permutation matrix, then $I(P, d) \geq d(k-1)$. Furthermore, if $P$ is the identity matrix, then $I(P, d) = S(P, d) = d(k-1)$.
\end{theo}
\begin{proof} We first show that, for all $n\ge k-1$, we have
\begin{equation}
\label{lb}
f(n, P, d) \geq n^d - (n - k +1)^d
\end{equation}
for every permutation matrix $P$. 
Pick one nonzero entry $p_{i_1, \ldots , i_d}=1$ of $P$. Construct a $d$-dimensional $n \times \cdots \times n$ matrix $A$ with entries such that $a_{j_1, \ldots , j_d}=0$ if $i_l \le j_l \le n-k+i_l$ for all $1 \leq l \leq d$ and $a_{j_1, \ldots , j_d}=1$ otherwise. We first show that $A$ avoids $P$. Suppose to the contrary that $A$ contains $P$. Let the special nonzero entry 
$p_{i_1, \ldots, i_d}=1$ of $P$ be represented by entry $a_{y_1, \ldots, y_d}$ of $A$.  By the construction of $A$, we must have 
either $y_l\le i_l-1$ or $y_l \ge n-k+i_l+1$. If $y_l \le i_l -1$, since $A$ contains $P$,
$A$ has $i_l -1$ other nonzero entries whose $l^{\text{th}}$ coordinates are smaller than $y_l \leq i_l -1$ to represent $1$-entries of $P$, an impossibility. If $y_l \geq n - k + i_l +1$, a similar argument leads to another impossibility. Counting the number of $1$-entries in $A$ proves (\ref{lb}).

We next show that 
\begin{equation}
\label{ub}
f(n,P,d) \le n^d-(n-k+1)^d
\end{equation}
when $P$ is the identity matrix, i.e.,   $p_{i_1, \ldots, i_d}$ is one on the main diagonal $i_1=\cdots=i_d$ and zero otherwise. 
If $A$ is a matrix that avoids $P$,  each diagonal of $A$, which is parallel to the main diagonal, has at most $k-1$ nonzero entries. Summing over the maximum numbers of $1$-entries in all diagonals
proves (\ref{ub}). 

The second part of Theorem \ref{llb} follows immediately from (\ref{lb}) and (\ref{ub}). The first part is obvious for a permutation matrix $P$ because of (\ref{lb}). 
The first part is also true 
for a tuple permutation matrix $P'$ since $f(n, P, d) \leq f(n, P', d)$ if $P'$ is generated by a permutation matrix $P$.
\end{proof}

The lower bound given in Theorem \ref{llb} is linear in $k$. One may ask how large a lower bound on $I(P,d)$ can be for some $P$. In the the rest of this section, we extend Fox's idea for the $d=2$ case \cite{Fox, Fox2} to show that a lower bound can be as large as an exponential function in $k$ in multiple dimensions. The crucial part in our approach is our observation that $f(n,P,d)$ is super-homogeneous.

\begin{theo}
\label{lower}
For each large $k$, there exists a family of $d$-dimensional $k\times \cdots \times k$ permutation matrices $P$ such that 
$I(P,d)=2^{\Omega(k^{1 / d})}$.
\end{theo}
The proof uses the super-homogeneity of extremal functions. In dimension two, the extremal function 
was shown to be super-additive \cite{PT}, i.e., $f(m+ n, P, 2) \geq f(m, P, 2) + f(n, P, 2)$. This was the key in showing the convergence of the sequence
$\{ {f(n,P, 2) \over n} \}$ for those  matrices $P$ whose extremal functions are $\Theta(n)$.  The  limit is the well-known F\"uredi-Hajnal limit \cite{FH}.

We note that the super-additivity of $f(n, P, 2)$ implies super-homogeneity, i.e., $f(s n, P, 2) \geq s f(n, P, 2)$ for every positive integer $s$. In higher dimensions, 
we show that $f(n, P, d)$ is
super-homogeneous of a higher degree.

A corner entry of a $k_1 \times \cdots \times k_d$ matrix $P = (p_{i_1, \ldots, i_d})$ is defined to be an  entry $p_{i_1, \ldots, i_d}$ located at a corner of $P$, i.e.,  where $i_{\tau}=1$ or $k_{\tau}$ for $1 \leq \tau \leq d$.

\begin{lem}
\label{homo}
If $P$ is a $d$-dimensional matrix with a corner $1$-entry, then $f(sn,P,d) \ge {s^{d-1} \over (d-1)!}f(n,P,d)$.
\end{lem}
\begin{proof}
 Without loss of generality, we assume that $p_{1, \ldots, 1}=1$ is the corner $1$-entry in $P$. Let $M$ be an $s \times \cdots\times s$ matrix with $1$-entries 
at the coordinates $(i_1, \ldots, i_d)$ where $i_1+ \cdots+i_d=s+d-1$ and $0$-entries everywhere else, so $M$ has ${s+d-2 \choose d-1} \ge {s^{d-1} \over (d-1)!}$ $1$-entries. 
Let $N$ be an $n \times \cdots \times n$ matrix that avoids $P$ and has $f(n,P,d)$ $1$-entries. It then suffices to prove that $M \otimes N$  avoids $P$.

Assume for contradiction that the Kronecker product $M \otimes N$ contains $P$. Pick an arbitrary $1$-entry $p^*$ in $P$ other than $p_{1, \ldots, 1}$.  Suppose that $p_{1, \ldots, 1}$ and $p^*$ are represented by $e_1$ and $e_2$  in $M \otimes N$, respectively.
We consider the $n \times \cdots \times n$ $S$ submatrices of $M \otimes N$. We may assume that $e_1$ and $e_2$ are in
the $S$-submatrices $S(i_1, \ldots, i_d)$ and $S(j_1, \ldots, j_d)$, respectively. Note that $i_1+ \cdots +i_d = j_1+ \cdots + j_d$. Since $p^*$
has larger coordinates than $p_{1, \ldots, 1}$ in $P$, entry $e_2$ must also have larger coordinates than $e_1$ in $M \otimes N$ and hence $i_{\tau} \leq j_{\tau}$ for $\tau = 1, 2, \ldots, d$. It then follows from $i_1+ \cdots +i_d = j_1+ \cdots + j_d$
that $i_{\tau} = j_{\tau}$ for $\tau= 1, 2, \ldots, d$, i.e., 
the two entries $e_1$ and $e_2$ must be in the same $S$ submatrix in $M \otimes N$. Since $p^*$ is an arbitrary $1$-entry other than $p_{1, \ldots , 1}$ in $P$, 
the $S$ submatrix contains $P$. But this is a contradiction since each nonzero $S$ submatrix in $M \otimes N$ is an exact copy of $N$, 
which avoids $P$. Thus $M \otimes N$ avoids $P$.
\end{proof}

Just as super-additivity leads to the F\"uredi-Hajnal limit in dimension two, super-homogeneity also produces an interesting result on limits.
\begin{lem}
\label{converge}
If $P$ is a $d$-dimensional matrix which contains a corner $1$-entry, then for any positive integer $m$,
$$I(P,d) \ge {1 \over (d-1)!} \ {f(m,P,d) \over m^{d-1}}.$$
\end{lem}
\begin{proof}
For each fixed positive integer $m$, we write $n$ as $n=sm+r$, where $0 \le r<m$. Then 
$${f(n,P,d) \over n^{d-1}} = {f(sm+r,P,d) \over (sm+r)^{d-1}} 
 \ge  {f(sm,P,d) \over (sm+r)^{d-1}} 
\ge {s^{d-1} \over (d-1)!}{f(m,P,d) \over (sm+r)^{d-1}} \ ,$$
where we use Lemma \ref{homo} in the second inequality. Now we take the limit inferior of the left side as $n$ goes to $\infty$. Since $m$ is fixed and $r$ is bounded, $s={n-r \over m}$ goes to $\infty$ as well. Hence,
$$ \liminf_{n \to \infty}{f(n,P,d) \over n^{d-1}} \ge \lim_{s \to \infty}{s^{d-1} \over (d-1)!}{f(m,P,d) \over (sm+r)^{d-1}} 
= {1 \over (d-1)!}{f(m,P,d) \over m^{d-1}} \ . $$
\end{proof}

The following lemma gives a lower bound on the right hand side of the inequality in Lemma \ref{converge} for a particular $m = N$. The proof is based on Fox's ideas in his lecture \cite{Fox2} for $d=2$ case.
\begin{lem}
\label{exists}
There exists a $d$-dimensional  $N \times \cdots \times N$ matrix $A$, where $N=2^{\Omega(\ell)}$, that has $\Theta(N^{d-1/2})$ $1$-entries and avoids $R^{\ell,\ldots, \ell}$ as an interval minor.
\end{lem}
\begin{proof}

We prove the lemma for $\ell$ that are multiples of $20$ and the result can then be easily extended to all $\ell$. Let $r={\ell \over 20}$, $q={\ln 2 \over 2r^{d-1}}$, and $N=2^r$. 

Define a dyadic interval to be a set of consecutive integers of the form $\{ (s-1)2^t+1, \ldots, s2^t \}$ for nonnegative integers $s$ and $t$, and a dyadic hyper-rectangle to be the Cartesian product of $d$ dyadic intervals. We consider only dyadic intervals which are subsets of $\{1, \ldots, N \}$.
Note that there are exactly $\sum_{t=0}^{r}2^t=2N-1$ such dyadic intervals and $(2N-1)^d$ such dyadic hyper-rectangles. There are $(r+1)^d$ dyadic hyper-rectangles 
containing each lattice point $(i_1, \ldots , i_d)$, where $i_j = 1, 2, \ldots , N$ and $j= 1, 2, \ldots , d$,  since each $i_j$ is contained in exactly $r+1$ dyadic intervals. 

Let $\cal R$ be a random collection of these dyadic hyper-rectangles, each included independently with probability $q$. Define $A$ to be the $N\times \cdots \times N$ $d$-dimensional matrix such that $a_{i_1, \ldots, i_d}=1$ if $(i_1, \ldots, i_d)$ is not contained in any dyadic hyper-rectangle of $\cal R$ and $a_{i_1, \ldots, i_d}=0$ otherwise. The expected number of $1$-entries in $A$ is 
$$(1-q)^{(r+1)^d}N^d =[(1 - q)^{1/q}]^{q (r+1)^d} N^d = \Theta(e^{-qr^d}N^d)=\Theta(2^{-r/2}N^d)=\Theta(N^{d-1/2}),$$
where we use $(1 - q)^{1/q} = \Theta(e^{-1})$ for small $q$ in the second  equality.

Denote by $X$ and $Y$ the events that $A$ contains and avoids $R^{\ell, \ldots, \ell}$ as an interval minor, respectively. We estimate the probability $\mathbb{P}(X)$. If $B$ is a set of dyadic intervals, let $\chi(B)$ be the number of dyadic intervals that contain at least one interval in $B$ as a subset. Then we define $h(x)$ to be the  number of sets $B$ containing $\ell$ dyadic intervals such that $\chi(B)=x$. 

If $A$ contains $R^{\ell, \ldots, \ell}$ as an interval minor, then there are  intervals of consecutive integers, denoted by $W_{i,1}, \ldots, W_{i, \ell}$, partitioning the set $\{ 1, 2, \ldots , N \}$  in the $i$th dimension such that every submatrix $W_{1, j_1} \times W_{2, j_2} \times \cdots \times W_{d, j_d}$ of $A$ contains at least one $1$-entry. We denote by $I_{i,j}$ the unique smallest-length dyadic interval which contains $W_{i,j}$. If there are $x_i$ dyadic intervals which contain at least one of the dyadic intervals in $B_i=\{I_{i,1}, \ldots, I_{i,\ell} \}$, then $x_i \geq \ell$ and there are $h(x_i)$ possible sets $B_i$. Thus there are $h(x_1) \cdots h(x_d)$ choices for $B_1, \ldots, B_d$ and there are $x_1 \cdots x_d$ dyadic hyper-rectangles which contain at least one of the $\ell^d$ dyadic hyper-rectangles of the form $I_{1, i_1}\times \cdots \times I_{d, i_d}$. Since these $x_1 \cdots x_d$ dyadic hyper-rectangles contain $1$-entries of $A$, none of them is in $\cal R$. Hence, the probability 
\begin{equation}
\label{P(X)}
\mathbb{P}(X)  \leq \sum_{x_1,\ldots,x_d\ge \ell} (1-q)^{x_1\cdots x_d}h(x_1)\cdots h(x_d).
\end{equation}

To find a bound on $h(x)$, we estimate the number of sets $U=\{u_1, \ldots, u_{\ell}\}$ of dyadic intervals such that $\chi(U)=x$. Let $v_1, \ldots, v_\ell$ be the integers such that $v_1=\chi(\{u_1 \})$ and $v_i=\chi(\{u_1, \ldots, u_i\})-\chi(\{u_1, \ldots, u_{i-1}\})$, for $i = 2, \ldots , \ell$. Since $v_1+\cdots+v_{\ell}=x$, there are at most ${x+\ell-1 \choose \ell-1}$ possible values for $v_1, \ldots, v_{\ell}$. Given $v_1, \ldots, v_{\ell}$, there are at most $2^{v_1}$ choices for $u_1$ and at most $2^{v_{i+1}}\chi(\{u_1, \ldots, u_i\})$ choices for $u_{i+1}$. Calculating the number of possible choices for $U$ gives
$$h(x) \le{{x+\ell-1 \choose \ell-1}2^{v_1}\prod_{i=1}^{\ell-1} 2^{v_{i+1}}\chi(\{u_1, \ldots, u_i\}) \over \ell !} \le {{x+\ell-1 \choose \ell-1}2^x x^{\ell}  \over \ell !}.$$ 
Substituting this into (\ref{P(X)}) yields
\begin{equation}
\mathbb{P}(X)  \leq \sum_{x_1,\ldots,x_d\ge \ell} (1-q)^{x_1\cdots x_d} \prod_{i=1}^{d} {{x_i+\ell-1 \choose \ell-1}2^{x_i} {x_i}^{\ell}  \over \ell !} \ .  \nonumber
\end{equation}
The summand 
 $$r(x_1,\ldots,x_d)=(1-q)^{x_1\cdots x_d}\prod_{i=1}^{d}{{x_i+\ell-1 \choose \ell-1}2^{x_i} x_i^{\ell}  \over \ell !}$$
is a symmetric function of its variables.
Note that the ratio ${r(x_1+1,\ldots, x_d) / r(x_1,\ldots,x_d)}$ is equal to 
$$2(1-q)^{x_2 \cdots x_d}{(x_1+\ell) \over x_1+1}(1+{1\over x_1})^{\ell} \le 4 e^{-qx_2\cdots x_d}(1+{1\over x_1})^{x_1}\le e^{-q\ell^{d-1}}4e=2^{-20^{d-1} / 2}4e \le 1/2 \ ,$$
where we use $(1 - q)^{1/q} \leq 1/e$ for small $q > 0$ and $(1 + 1/x_1)^{x_1} \leq e$ for $x_1 \geq 1$. We also note that $$r(\ell, \ldots , \ell) = (1-q)^{\ell^d}\left[{2\ell-1\choose \ell-1}2^{\ell}{\ell^{\ell} \over \ell!} \right]^d  \le  e^{-q\ell^d}(2^{2\ell-1}2^{\ell}e^{\ell})^d \le 
2^{-20^{d-1}\ell/2 } (2^{3\ell - 1}e^{\ell})^d \le (2N)^{-d}, $$
where we use Stirling's inequality and ${2\ell-1 \choose \ell-1}\le 2^{2\ell-1}$.
We now use the 
symmetry of $r(x_1, \cdots, x_d)$ to obtain
$$ \mathbb{P}(X)  \le \sum_{x_1, \ldots, x_d\ge \ell} r(x_1,\ldots, x_d)  \le \left(\sum_{i=0}^{\infty}(1/2)^i \right)^d r(\ell,\ldots,\ell) 
\le 2^d  (2N)^{-d} = N^{-d}. $$

We now estimate conditional expectation $\mathbb{E}( \xi | Y)$, where $\xi = |A|$.  Note that $\mathbb{E}(\xi | Y)\mathbb{P}(Y)=\mathbb{E}(\xi)-\mathbb{E}(\xi | X) \mathbb{P}(X) \ge \Theta(N^{d-1/2})-N^d N^{-d} = \Theta(N^{d-1/2})$ so $\mathbb{E}(\xi | Y)=\Theta(N^{d-1/2})$.
Thus, there exists an $A$ that avoids $R^{l, \ldots,l}$ as an interval minor and has at least $\Theta(N^{d-1/2})$ $1$-entries.
\end{proof}

We are now ready to prove Theorem \ref{lower}.
\begin{proof} [Proof of Theorem \ref{lower}]

Let $\ell=\lfloor k^{1/d} \rfloor$ be the largest integer less than or equal to $ k^{1/d}$. There is a family  of $d$-dimensional permutation matrices of size $\ell^d \times \cdots \times \ell^d$ that contain $R^{\ell, \ldots , \ell}$ as an interval minor and have at least one corner $1$-entry. To see this, 
many such permutation matrices can be constructed so that they have exactly one $1$-entry in each of their $S$-submatrices of size $\ell^{d-1} \times \cdots \times \ell^{d-1}$, including a corner $1$-entry.

Since $k \geq \ell^d$, there is a family of  permutation matrices $P$ of size $k \times \cdots \times k$ that contain $R^{\ell, \ldots , \ell}$ as an interval minor and have at least one corner $1$-entry.
Each $P$ has a corner $1$-entry, so we can apply Lemma \ref{converge} to obtain
\begin{equation}
\label{key}
I(P,d) \ge {1 \over (d-1)!} \ {f(N,P,d) \over N^{d-1}} ,
\end{equation}
where $N$ can be chosen to be the positive integer given in Lemma 4.9. 

Matrix $P$ contains $R^{\ell,\ldots,\ell}$ as an interval minor, 
so $f(N,P,d)\ge m(N,R^{\ell, \ldots, \ell},d)$,
which along with (\ref{key}) and Lemma \ref{exists} yields 
$$I(P,d)  \ge {1 \over (d-1)!} \ {m(N, R^{\ell, \ldots, \ell}, d) \over N^{d - 1}} 
 \geq \Theta(N^{1 \over 2}) 
= 2^{\Omega(\ell)}
= 2^{\Omega(k^{1 / d})}.
$$
This completes the proof of Theorem \ref{lower}.
\end{proof}

\section{Conclusions and future directions}
\label{conclusion}
We obtained non-trivial lower and upper bound on $f(n,P,d)$ when $n$ is large for block permutation matrices $P$. In particular, we established the tight bound $\Theta(n^{d-1})$ on $f(n,P,d)$ for every $d$-dimensional tuple permutation matrix $P$. 
We improved the previous upper bound on the limit superior of the sequence $\{ {f(n,P,d) \over n^{d-1}}\}$ for 
all permutation and tuple permutation matrices. We used the super-homogeneity of the extremal function to show that 
the limit inferior is exponential in $k$ for a family of $k \times \cdots \times k$ permutation matrices. Our results substantially advance the extremal theory of matrices. We believe that super-homogeneity is fundamental to pattern avoidance in multidimensional matrices.

One possible direction for future research would be to strengthen the super-homogeneity as expressed in Lemma \ref{homo} to
$f(sn, P, d) \geq s^{d-1} f(n, P, d)$. We have successfully tested this super-homogeneity on the identity matrix and the matrices whose $1$-entries are on rectilinear paths. If this super-homogeneity is true 
for permutation matrices $P$, we can then use a Fekete-like lemma to show the convergence of the sequence $\{ {f(n,P,d) \over n^{d-1}} \}$.

Another possible direction would be to extend Theorem \ref{lower} from a family of permutation matrices to almost all permutation matrices. We think this becomes possible if  the corner $1$-entry condition is removed in Lemmas \ref{homo} and \ref{converge}.

\section{Acknowledgments}

We would like to thank Professor Jacob Fox for valuable discussions about this research. The first author was supported by the NSF graduate fellowship under grant number 1122374. The research of the second author was supported in part by the Department of Mathematics, MIT through PRIMES-USA 2014. It was also supported in part  by the Center for Excellence in Education and the Department of Defense through RSI 2014.

\bibliographystyle{amsplain}

\end{document}